\documentclass{amsart} 

\usepackage{amsmath,amssymb,amsthm,amsfonts,amsxtra,mathtools}
\usepackage{enumerate,verbatim,mathrsfs,comment,color,url}
\usepackage[all,2cell,ps]{xy}
\usepackage[pagebackref]{hyperref}

\usepackage{graphicx}
\usepackage{tikz-cd}
\usepackage{verbatim}
\usepackage{stackengine}
\usepackage{cleveref}
\usepackage{array}
\usepackage{booktabs}
\usepackage[dvipsnames]{xcolor}

\DeclareMathOperator{\Ass}{Ass}

\DeclareMathOperator{\depth}{depth}
\DeclareMathOperator{\eend}{end}

\DeclareMathOperator{\Ext}{Ext}

\DeclareMathOperator{\Hom}{Hom}

\DeclareMathOperator{\indeg}{indeg}

\DeclareMathOperator{\reg}{reg}

\DeclareMathOperator{\Tor}{Tor}

\renewcommand{\ge}{\geqslant}
\renewcommand{\le}{\leqslant}

\newcommand{\bn}{\mathbb{N}}
\newcommand{\bz}{\mathbb{Z}}

\newcommand{\lra}{\longrightarrow}
\newcommand{\xlra}{\xlongrightarrow}

\newcommand{\mscrm}{\mathscr{M}}

\newcommand{\uc}{\underline{c}}
\newcommand{\um}{\underline{m}}
\newcommand{\un}{\underline{n}}
\newcommand{\ur}{\underline{r}}

\theoremstyle{plain}
\newtheorem{theorem}{Theorem}[section]
\newtheorem{lemma}[theorem]{Lemma}
\newtheorem{proposition}[theorem]{Proposition}
\newtheorem{corollary}[theorem]{Corollary}

\theoremstyle{definition}
\newtheorem{definition}[theorem]{Definition}

\newtheorem{example}[theorem]{Example}

\newtheorem{notation}[theorem]{Notation}

\newtheorem{para}[theorem]{}
\newtheorem{question}[theorem]{Question}
\newtheorem{setup}[theorem]{Setup}

\theoremstyle{remark}
\newtheorem{remark}[theorem]{Remark}

\numberwithin{equation}{theorem}

\title[Asymptotic regularity of (co)homology modules]{Asymptotic Castelnuovo-Mumford regularity of (co)homology modules involving powers of ideals}
 
\author[D.~Ghosh]{Dipankar Ghosh}
\address{(Ghosh) Department of Mathematics, Indian Institute of Technology Kharagpur, West Bengal - 721302, India}
\email{dipankar@maths.iitkgp.ac.in, dipug23@gmail.com}
\urladdr{\url{https://orcid.org/0000-0002-3773-4003}}

\author[R.~Nanduri]{Ramakrishna Nanduri}
\address{(Nanduri) Department of Mathematics, Indian Institute of Technology Kharagpur, West Bengal - 721302, India}
\email{nanduri@maths.iitkgp.ac.in}
\urladdr{\url{https://orcid.org/0000-0002-5136-8214}}

\author[S.~Pramanik]{Siddhartha Pramanik}
\address{(Pramanik) Department of Mathematics, Indian Institute of Technology Kharagpur, West Bengal - 721302, India}
\email{siddharthap@kgpian.iitkgp.ac.in, pramaniksiddhartha2@gmail.com}
\urladdr{\url{https://orcid.org/0009-0005-1219-4265}}

\date{October 12, 2026}

\begin{document}
    \pagenumbering{arabic}
    \thispagestyle{empty}

\dedicatory{Dedicated to Professor Winfried Bruns and Professor J\"urgen Herzog in recognition of their profound contributions to commutative algebra.}

\begin{abstract}
%
Let $R$ be a Noetherian standard $\mathbb{N}$-graded ring, and let $I_1,\ldots,I_r$ be homogeneous ideals of $R$. Let $L,M$, and $N$ be finitely generated $\mathbb{Z}$-graded $R$-modules with $N\subseteq M$. For $\underline{n} :=(n_1,\dots,n_r)\in\mathbb N^r$, set ${\bf I}^{\underline{n}}:=I_1^{n_1}\cdots I_r^{n_r}$. We study the asymptotic behaviour of the (Castelnuovo-Mumford) regularity of two families of Ext and Tor modules: those obtained from the quotients ${\bf I}^{\underline{n}}M/{\bf I}^{\underline{n}}N$, and those obtained from the quotients $M/{\bf I}^{\underline{n}}N$, with $L$ as the first argument in both the families. For each fixed homological degree, we prove that the regularity of the former family is eventually given by the supremum of finitely many linear functions of $\underline{n}$, where the coefficient of $n_i$ belongs to the set of degrees of generators of $I_i$. For the second family, when at least one of the ideals $I_i$ is generated by homogeneous elements of positive degree, we establish, under suitable conditions, a dichotomy: the regularity is either eventually equal to the regularity of the Ext or Tor module of $(L,M)$, or exhibits the same asymptotic behaviour as in the first family. These results follow from a theorem describing the asymptotic behaviour of the regularity of $(U+{\bf I}^{\underline{n}}V)/{\bf I}^{\underline{n}}W$, under suitable conditions, where $U$, $V$, and $W$ are graded submodules of a finitely generated graded $R$-module with $W\subseteq V$.
\end{abstract}

\maketitle 

\section{Introduction}

The (Castelnuovo-Mumford) regularity (Definition~\ref{def:reg}) is a fundamental invariant in Commutative Algebra and Algebraic Geometry, measuring the complexity of a graded module through the vanishing of the graded components of its local cohomology modules. Let $R$ be a standard graded polynomial ring over a field, and $I$ a homogeneous ideal of $R$. Cutkosky-Herzog-Trung \cite[Thm.~1.1.(ii)]{CHT99} and independently Kodiyalam \cite[Thm.~5]{Ko00} proved that $\reg(I^n)$ is eventually a linear function of $n$, i.e., there exist constants $a$ and $b$ such that $\reg(I^n) = an+b$ for all $n\gg 0$. This result was extended by Trung-Wang \cite[Thm.~3.2]{TW05}, who showed that, for a finitely generated graded $R$-module $M$, $\reg(I^nM)$ is asymptotically linear in $n$, where $R$ is a Noetherian standard $\bn$-graded ring. For powers and products of several ideals, Ghosh obtained a linear upper bound of this function, see \cite[Cor.~4.4]{Gh16}. Subsequently, Bruns-Conca \cite[Thm.~2.2]{BC17} proved that $\reg(I_1^{n_1}\cdots I_r^{n_r})$ is eventually given by the maximum of finitely many linear functions of $n_1, \dots, n_r$, where $I_1, \dots, I_r$ are homogeneous ideals of $R$. We refer the reader to \cite[Thm.~5.4]{GP19} and \cite[Thm.~A]{CGN22} for further variations and related results. In the present article, we prove Proposition~\ref{prop:regLM}, which unifies and highly extends the results of Trung-Wang, Ghosh, and Bruns-Conca. More precisely, for a finitely generated graded $R$-module $L$, under the assumption that at least one of the ideals $I_i$ is generated in positive degrees, we show that the regularity of $L+I_1^{n_1}\cdots I_r^{n_r}M$ is eventually either equal to $\reg(L)$ or given by the supremum of finitely many linear functions of $(n_1,\ldots,n_r)$, with the coefficient of $n_i$ belonging to the set of degrees of generators of $I_i$.

The asymptotic behaviour of other algebraic invariants associated with Ext and Tor modules has also been extensively studied. Let $k \ge 0$ be a fixed integer. In \cite[Thm.~1]{MS93}, Melkersson-Schenzel proved the asymptotic stability of the set $\Ass_R(\Tor_k^R(L, R/I^n))$ of associated primes, where $R$ is a Noetherian ring, $I$ is an ideal of $R$, and $L$ is a finitely generated $R$-module. More generally, in \cite[Cor.~3.5]{KW04}, Katz-West showed that the sets $\Ass_R(\Ext^k_R(L,M/I^nN))$ and $\Ass_R(\Tor_k^R(L,M/I^nN))$ become stable for all $n \gg 0$, where $N \subseteq M$ are finitely generated $R$-modules. See \cite{GP19b} for a concise proof of this result. Moreover, for an ideal $J$ of $R$, it follows from a result \cite[Thm.~1.1]{S17} of Se that $\depth_J(\Ext^k_R(L,M/I^nM))$ and $\depth_J(\Tor_k^R(L,M/I^nM))$ stabilize asymptotically. Banda-Melkersson gave a different proof of this result in \cite[Thm.~3.5]{BM2019}. In another direction, Kodiyalam \cite[Thm.~2]{Ko93} proved that the length functions
$$n \mapsto \lambda_R(\Ext^k_R(L, M/I^nM)) \quad \mbox{and} \quad n \mapsto \lambda_R(\Tor_k^R(L, M/I^nM))$$
are eventually given by polynomials, provided $\lambda_R(L \otimes_RM) < \infty$. Theodorescu \cite[Cor.~4]{T02} obtained the same conclusion under a weaker hypothesis. See \cite[Thm.~1.3]{DGPPS25} for a common generalization of the aforementioned results in the case of powers and products of several ideals. In the graded setting, Ghosh-Pramanik \cite[Thm.~1.3]{GP25} proved that the Vasconcelos invariants $v(\Ext^k_R(L, M/I^nN))$ and $v(\Tor_k^R(L, M/I^nN))$ are eventually linear in $n$, under certain conditions. These results demonstrate that several algebraic invariants of the modules $\Ext_R^k(L, M/ I^nN)$ and $\Tor_k^R(L, M/I^n N)$ exhibit good asymptotic behaviour. Motivated by these developments, it is natural to investigate the asymptotic behaviour of the regularity of the modules $\Ext_R^k(L, M/ I^nN)$ and $\Tor_k^R(L, M/I^n N)$ as functions of $n$. 

Let $R$ be a Noetherian standard $\bn$-graded ring over an Artinian local ring, and $I_1, \dots, I_r$ be homogeneous ideals of $R$. Let $L$, $M$ and $N$ be finitely generated graded $R$-modules, with $N \subseteq M$. Then it follows from \cite[Thm.~4.3]{Gh16} that for each fixed $k \ge 0$, the functions
$$ \reg\big(\Ext^k_R(L,M/I_1^{n_1}\cdots I_r^{n_r} N)\big) \; \quad\text{and}\quad \reg\big(\Tor_k^R(L,M/I_1^{n_1}\cdots I_r^{n_r}N)\big)$$
are bounded above by linear functions of $n_1, \dots, n_r$. This naturally leads to the question that what is the exact asymptotic behaviour of each of these functions. We are not aware of a result giving an exact asymptotic description of these regularity functions, even in the case of powers of a single ideal. The goal of this article is to investigate this question. Before stating the main results, we fix the following setup and notation for the remainder of this paper.

\begin{setup}\label{setup1}
Let $R := R_{0}[x_1,\dots,x_d]$ be a Noetherian standard $\bn$-graded ring. Set $R_+ := \langle x_1,\dots,x_d \rangle$. Let $I_1,\dots,I_r$ be homogeneous ideals of $R$. Suppose that $I_i$ is generated by homogeneous elements $f_{i,1},\dots,f_{i,g_i}$ with $\deg(f_{i,j}) = d_{i,j}$ for $1\le i\le r$ and $1\le j\le g_i$. Let $L$, $M$ and $N$ be finitely generated $\bz$-graded $R$-modules such that $N\subseteq M$. Denote ${\bf I}^{\un} := I_1^{n_1} \cdots I_r^{n_r}$, where $\underline{n} =(n_1, \dots, n_r) \in \bn^r$.
\end{setup}

\begin{notation}
Let $r\ge 1$ be an integer. A typical element of $\bz^r$ is denoted by $\un:=(n_1, \dots, n_r)$. Set $\underline{0} := (0,\ldots,0)$. For $1 \le i \le r$, let $e_i$ denotes the $i$th basis vector of $\bn^r$.
For $\um, \un\in\bz^r$, define $\um \ge \un$ if $m_i \ge n_i$ for all $1\le i\le r$. By `for all $\un \gg 0$', we mean `for all $\un \ge \uc$' for some $\uc \in \bn^r$. The \emph{initial degree} of $M$ is defined as $\indeg(M) := \inf \{ n \in \bz : M_n \neq 0\}$, with the convention $\indeg(0) := \infty$.
\end{notation}

In this article, we mainly prove the following theorem. 

\begin{theorem}[See Theorems~\ref{th:mainext} and \ref{th:maintor}]\label{th:main}
With \textnormal{Setup~\ref{setup1}}, let $\indeg(I_i) >0$ for some $i$. Fix $k \ge 0$. Set $G := \Ext_R^k(L,M)$ and $G_{\un} := \Ext_R^k(L, M/{\bf I}^{\un} N)$; or $G := \Tor_k^R(L,M)$ and $G_{\un} := \Tor_k^R(L,M/{\bf I}^{\un}N)$, for $\un \ge \underline{0}$. Assume that at least one of the following two conditions holds:
    \begin{enumerate}[\rm (i)]
    \item $G$ is $R_+$-torsion.
    \item $G_{\un}$ is $R_+$-torsion for all $\un \gg 0$.
    \end{enumerate}
    
Then, either $G_{\un} = 0$ for all $\un \gg 0$, or $G_{\un} \neq 0$ for all $\un \gg 0$. Moreover, in the second case $($i.e., when $G_{\un} \neq 0$ for all $\un \gg 0$$)$, exactly one of the following holds:
\begin{enumerate}[\rm (1)]
    \item $\reg(G_{\un}) = \reg(G)$ for all $\un \gg 0$.
    \item $\reg(G_{\un}) = \sup \{ \mathscr{L}_{s}(\un) : 1 \le s \le m\}$ for all $\un \gg 0$,
    where each $\mathscr{L}_s : \bn^r \to \bz$ is a linear function defined as $$\mathscr{L}_s(\underline{n}) = (w_{s,1}n_1 + w_{s,2}n_2 + \cdots + {w_{s,r}n_r}) + c_{s}$$ 
    for some $w_{s,i}\in \{d_{i,1},\dots,d_{i,g_{i}}\}$ and $c_{s} \in \bz$ for $1 \le s \le m$ and $1\le i\le r$.
\end{enumerate}
\end{theorem}

Note that in Theorem~\ref{th:main}, if $L$ or $M$ is $R_+$-torsion, then so is $G$. Also, if $M/N$ is $R_+$-torsion and $R_+ \subseteq \sqrt{I_i}$ for each $i$, then $G_{\un}$ is $R_+$-torsion for all $\un \ge \underline{0}$. In the case when $G=0$ and $G_{\un} \neq 0$ for all $\un \gg 0$, it follows from Theorem~\ref{th:main} that $\reg(G_{\un})$ is eventually the supremum of finitely many linear functions of $\un$. In the special case where each ideal $I_i$ is generated in a single degree, say $d_i$, for $1 \le i \le r$, we have either $\reg(G_{\un}) = \reg(G)$ for all $\un \gg 0$, or $\reg(G_{\un})$ is eventually a linear function of $\un$, where the coefficient of $n_i$ is $d_i$. It is worth mentioning the special case of a single ideal separately as a corollary, where $\reg(G_n)$ is eventually either constant or linear in $n$. Since the maximum of finitely many linear functions of $n$ is eventually a linear function (the one with largest slope, and among equal slopes the largest intercept), the following corollary is immediate from Theorem~\ref{th:main}.

\begin{corollary}\label{cor:main}
    Let $R$, $L$, $M$ and $N$ be as in \textnormal{Setup~\ref{setup1}}. Let $I$ be a homogeneous ideal with $\indeg(I)>0$. Fix $k \ge 0$. Set $G := \Ext_R^k(L,M)$ and $G_n := \Ext_R^k(L, M/I^n N)$; or $G := \Tor_k^R(L,M)$ and $G_n := \Tor_k^R(L,M/I^nN)$, for $n \ge 0$. Assume that $G$ is $R_+$-torsion, or $G_n$ is $R_+$-torsion for all $n \gg 0$.
    
Then, either $G_n = 0$ for all $n \gg 0$, or $G_n \neq 0$ for all $n \gg 0$. Moreover, in the second case, exactly one of the following holds:
\begin{enumerate}[\rm (1)]
    \item $\reg(G_n) = \reg(G)$ for all $n \gg 0$.
    \item $\reg(G_n) = w n + c$ for all $n \gg 0$, where $c \in \bz$ and $w$ is one of the generating degrees of the ideal $I$.
\end{enumerate}
\end{corollary}

We provide some examples to illustrate the main results. Example~\ref{example2} shows that both conclusions (1) and (2) of Theorem~\ref{th:main} and Corollary~\ref{cor:main} can occur, whereas Example~\ref{exam:non-linear}.(1) shows that, in the case of several ideals, these functions need not be asymptotically linear.

We organize the latter sections of this article as follows. In Section~\ref{sec2}, we recall some definitions and known results. In Section~\ref{sec3}, for each fixed $k\ge 0$, we investigate the asymptotic behaviour of regularity of $\Ext^k_R(L,{\bf I}^{\un}M/{\bf I}^{\un}N)$ and $\Tor_k^R(L,{\bf I}^{\un}M/{\bf I}^{\un}N)$ as functions of $\un$. In Section~\ref{sec4}, we prove Theorem~\ref{th:main} by establishing an asymptotic formula for $\reg((U + {\bf I}^{\un}V)/{\bf I}^{\un}W)$ under certain conditions in Theorem~\ref{th:main1}. Finally, in Section~\ref{sec:examples}, we provide concrete examples to complement the main results. 

\subsection*{Acknowledgments}
Pramanik received financial support from the Government of India through the Prime Minister's Research Fellowship (PMRF) for his Ph.D.


\section{Preliminaries}\label{sec2}

In this section, we fix the notation, recall the definition and basic properties of Castelnuovo-Mumford regularity, and gather the results needed in the sequel. Throughout this section, we assume Setup~\ref{setup1}. For each $i \ge 0$, let $H^i_{R_+}(M)$ denote the $i$th local cohomology module of $M$ with support in $R_+$. For $n \in \bz$, we denote by $M_n$ the $n$th graded component of $M$. Denote $$\eend(M) := \sup\{ n \in \bz : M_n \neq 0\},$$ 
with the convention that $\eend(0) = - \infty$. 

\begin{definition}\label{def:reg}
The \emph{Castelnuovo-Mumford regularity} (or simply \emph{regularity}) of an $R$-module $M$ is defined as 
$$ \reg_R(M) := \sup\{ \eend(H^i_{R_+}(M))+i : i \ge 0 \},$$
with convention $\reg_R(0) = -\infty$. We write $\reg(M)$ instead of $\reg_R(M)$ when the base ring $R$ is clear from the context.
\end{definition}

We gather the following basic properties of regularity.

\begin{lemma}\label{lem:reglem}
    Let $0\to A \to B \to C \to 0$ be a short exact sequence of finitely generated graded $R$-modules. Then the following hold.
    \begin{enumerate}[\rm (1)]
    \item $\reg(A) \le \max\{\reg(B), \reg(C) + 1\}$.
    \item $\reg(B) \le \max\{\reg(A), \reg(C) \}$. The inequality becomes equality whenever the short exact sequence splits.
    \item $\reg(C )\le \max\{\reg(A) -1, \reg(B)\}$.
    \item{\cite[Prop.~8.1.11]{BCRV22}} If $A$ is an $R_+$-torsion module, then 
    $$\reg(B) = \max \{ \reg(A), \reg(C) \}.$$
    \end{enumerate}
\end{lemma}




For a finitely generated $\bz^r$-graded module $\mscrm$ over a Noetherian standard $\bn^r$-graded ring $S$, it follows from {\cite[Thm.~3.4.(i)]{W04}} that $\Ass_{S_{\underline{0}}}(\mscrm_{\un})$ becomes independent of $\un$ for all $\un \gg 0$. Since $\mscrm_{\un} = 0$ if and only if $\Ass_{S_{\underline{0}}}(\mscrm_{\un}) = \emptyset$, we obtain the following consequence.

\begin{lemma}\label{lem:stab}
Let $\mscrm$ be a finitely generated $\bz^r$-graded module over a Noetherian standard $\bn^r$-graded ring $S$. Then either $\mscrm_{\un} = 0 $ for all $\un \gg 0$, or $ \mscrm_{\un} \neq 0$ for all $\un \gg 0$.
\end{lemma}

In the following theorem, we prove the asymptotic behaviour of the regularity of the graded components of a multigraded module. As we could not find any reference for this result in the literature, we include a sketch of its proof. The bigraded version of this result can be found in \cite[Thm.~2.2]{TW05} or \cite[Prop.~3.1]{CGN22}.

\begin{theorem}\label{th:reggraded}
Let $S = R_0[x_1, \dots, x_d, z_{i,j} : 1 \le i \le r , 1 \le j \le g_i]$ be a $\bn \times \bn^r$-graded Noetherian ring, where $\deg(x_i) = (1, \underline{0})$ and $\deg(z_{i,j})= (d_{i,j},e_i)$. Let $\mathscr{M}$ be a finitely generated $\bz \times \bz^r$-graded $S$-module. Set $R := R_0[x_1, \dots, x_d]$, where $\deg(x_i) = 1$. For each $\underline{n} \in \bz^r$, denote $\mathscr{M}_{(\star,\underline{n})} := \bigoplus_{l \in \bz}\mathscr{M}_{(l,\un)}$. Then, either ${\mathscr{M}_{(\star,\underline{n})}} = 0$ for all $\un \gg 0$, or  
\[
\reg_R\big({\mathscr{M}_{(\star,\underline{n})}}\big) = \sup\big\{ \mathscr{L}_{s}(\un) : 1 \le s \le m\big\} \; \mbox{ for all } \; \un \gg 0,
\]
where each $\mathscr{L}_s : \bn^r \to \bz$ is a linear function defined as $$\mathscr{L}_s(\underline{n}) = (w_{s,1}n_1 + w_{s,2}n_2 + \cdots + {w_{s,r}n_r}) + c_{s}$$ 
for some $w_{s,i}\in \{ d_{i,1},\dots,d_{i,g_{i}}\}$ and $c_{s} \in \bz$ for $1 \le s \le m$ and $1\le i\le r$. 
\end{theorem}

\begin{proof}
    Consider the polynomial ring $A := R_0[X_1, \dots, X_d] $ with standard $\bn$-graded structure. Then there is a surjective $R_0$-algebra homomorphism $A \to R$, which maps $X_i \mapsto x_i$ for each $i$. Through this homomorphism, every finitely generated graded $R$-module becomes a finitely generated graded $A$-module. Then, for any finitely generated graded $R$-module $M$, since $H_{R_+}^i(M) \cong H_{A_+}^i(M)$ for all $i$, we have $\reg_R(M) = \reg_A(M)$.
    Note that
    \begin{equation}\label{eq:regtorint}
    \reg_A(M) 
    = \max \{ \eend(\Tor_k^A(M,R_0)) -k : 0 \le k \le d\},
    \end{equation} 
    see, e.g., \cite[pp.~279 and Thm.~8.1.3]{BCRV22}. Also, by the graded version of Nakayama's lemma, $\Tor_0^A(M,R_0) = 0$ if and only if $M=0$. By Lemma~\ref{lem:stab}, either ${\mathscr{M}_{(\star,\underline{n})}} = 0$ for all $\un \gg 0$, or ${\mathscr{M}_{(\star,\underline{n})}} \neq 0$ for all $\un \gg 0$. We may assume that ${\mathscr{M}_{(\star,\underline{n})}} \neq 0$ for all $\un \gg 0$.
    
    Consider the $\bn \times \bn^r$-graded polynomial ring $A' := A[Z_{i,j} :1 \le i \le r , 1 \le j \le g_i]$, where $\deg(X_i) = (1, \underline{0})$ and $\deg(Z_{ij})= (d_{i,j},e_i)$. Then we have an $R_0$-algebra map $A' \to S$ which sends $X_t \mapsto x_t$ and $Z_{i,j} \mapsto z_{i,j}$. Moreover $\mscrm$ becomes a finitely generated $\bz \times \bz^r$-graded $A'$-module. Consequently, 
    $$\Tor_k^{A}(\mscrm, R_0) = \bigoplus_{\un \in \bz^r} \Tor_k^{A}(\mscrm_{(\star,\un)}, R_0)$$
    is a finitely generated $\bz^r$-graded module over the Noetherian standard $\bn^r$-graded ring $A' = A[Z_{i,j} :1 \le i \le r , 1 \le j \le g_i]$, where $\deg(Z_{i,j}) = e_i$. Let $0 \le k \le d$. In view of \cite[Thm.~8.3.7]{BCRV22}, either $\Tor_k^{A}(\mscrm_{(\star,\un)}, R_0) = 0$ for all $\un \gg 0$, or
    \begin{equation}\label{eq:endtor}
    \eend(\Tor_k^{A}(\mscrm_{(\star,\un)}, R_0)) = \sup \{ \mathscr{L}_{s}(\un) : \; 1\le s \le m\} \; \text{ for all } \un \gg 0,
    \end{equation} 
    where $\mathscr{L}_1, \dots, \mathscr{L}_m : \bz^r \to \bz$ are linear functions defined as 
    $$\mathscr{L}_s(\underline{n}) = (w_{s,1}n_1 + w_{s,2}n_2 + \cdots + {w_{s,r}n_r}) + c_{s},$$ 
    for some $w_{s,i}\in \{ d_{i,1},\dots,d_{i,g_{i}}\}$ and $c_{s} \in \bz$ for $1 \le s \le m$ and $1\le i\le r$. Thus, our claim follows from \eqref{eq:regtorint} and \eqref{eq:endtor}.
\end{proof}




The following lemma can be extracted from the proof of \cite[Cor.~3.4]{KW04}. We also refer the reader to {\cite[Lem.~3.3]{GP25}}, where a detailed proof of the corresponding result for a single ideal is given, together with the construction of the module $U$, which is crucial for the proof of our main result. The same arguments, together with the multigraded version of Artin-Rees lemma (cf.~\cite[Thm.~17.1.6]{SH06}), extend naturally to the case of several ideals.

\begin{lemma}\label{lem:ext:interpret}
With \textnormal{Setup~\ref{setup1}}, fix $ k \ge 0$. Then, there exists $\uc \in \bn^r$ such that
$$ \Ext_{R}^{k}(L, M/{\bf I}^{\un}N) \cong (U+{\bf I}^{\un - \uc}V)/{\bf I}^{\un - \uc} W \; \text{ for all } \un \ge \uc$$
and for some graded submodules $U$, $V$ and $W$ of a finitely generated graded $R$-module $T$ satisfying $W \subseteq V$ and $U\cong \Ext_R^k(L,M)$.
\end{lemma}

A similar interpretation of $\Tor^R_k(L,M/{\bf I}^{\un}N)$ also exists, see \cite[Lem.~3.4]{GP25} for the single ideal case.

\begin{lemma}\label{lem:tor:interpret}
With \textnormal{Setup~\ref{setup1}}, fix $ k \ge 0$. Then, there exists $\uc \in \bn^r$ such that
$$ \Tor_{k}^{R}(L, M/{\bf I}^{\un}N) \cong (U+{\bf I}^{\un - \uc}V)/{\bf I}^{\un - \uc} W \; \text{ for all } \un \ge \uc$$
and for some graded submodules $U$, $V$ and $W$ of a finitely generated graded $R$-module $T$ satisfying $W \subseteq V$ and $U\cong \Tor_k^R(L,M)$.
\end{lemma}

In view of Lemmas~\ref{lem:ext:interpret} and \ref{lem:tor:interpret}, to prove Theorem~\ref{th:main}, it is sufficient to investigate the asymptotic behaviour of $\reg\big((U+{\bf I}^{\un}V)/{\bf I}^{\un}W\big)$ as a function of $\un$, where $U$, $V$ and $W$ are graded submodules of a finitely generated $\bz$-graded $R$-module $T$ satisfying $W \subseteq V$.

\section{Asymptotic regularity of $\Ext^k_R(L,{\bf I}^{\un}M/{\bf I}^{\un}N)$ and $\Tor_k^R(L,{\bf I}^{\un}M/{\bf I}^{\un}N)$}\label{sec3}

In this section, we investigate the asymptotic behaviour of regularity of the modules $\Ext_R^k(L,{\bf I}^{\un}M/{\bf I}^{\un}N)$ and $\Tor_k^R(L,{\bf I}^{\un}M/{\bf I}^{\un}N)$ as functions of $\un$ for every fixed $k\ge 0$. Throughout this section, we assume Setup~\ref{setup1}.
    
\begin{para}\label{para:stability}
    Let $\mathcal{R}({\bf I}) := \bigoplus_{\un\in \bn^r} {\bf I}^{\un}$ be the multi-Rees algebra of $I_1, \dots, I_r$. Then the multi-Rees module $\mathcal{R}({\bf I},M) := \bigoplus_{\un \in \bn^r} {\bf I}^{\un} M$ is a finitely generated $\bn^r$-graded $\mathcal{R}({\bf I})$-module. Clearly, $\mathcal{R}({\bf I},N) = \bigoplus_{\un \in \bn^r}{\bf I}^{\un}N$ is a $\bn^r$-graded $\mathcal{R}({\bf I})$-submodule of $\mathcal{R}({\bf I},M)$. So the quotient $\frac{\mathcal{R}({\bf I},M)}{\mathcal{R}({\bf I},N)} = \bigoplus_{\un \in \bn^r} \frac{{\bf I}^{n}M}{{\bf I}^{n}N}$ is a finitely generated $\bn^r$-graded $\mathcal{R}({\bf I})$-module.
\end{para}

\begin{para}[{\bf $\bz^r$-graded structure on Ext and Tor}]
     Let $\mathcal{H}:=\bigoplus_{\un \in \bz^r} H_{\un}$ be a finitely generated $\bz^r$-graded module over the $\bn^r$-graded ring $\mathcal{R}({\bf I})$. Let 
     $$\mathbb{F}: \cdots \to F_n \to F_{n-1} \to \cdots \to F_1 \to F_0 \to 0$$ 
     be a free resolution of $L$ over $R$ consisting of finitely generated free $R$-modules. Then for each $k \ge 0$, the module $\Ext^k_R(L,\mathcal{H})$ is a subquotient of $\Hom_R(F_k, \mathcal{H})$. Hence $\Ext^k_R(L,\mathcal{H})$ is a finitely generated $\mathcal{R}({\bf I})$-module. Now, for $a \in {\bf I}^{\ur}$, the map $ H_{\un} \xlra{a} H_{\un+\ur}$ induces a map $\Hom_R(\mathbb{F}, H_{\un}) \xlra{a} \Hom_R(\mathbb{F}, H_{\un + \ur})$. It induces a map on the cohomology level $\Ext_R^k(L,H_{\un}) \xlra{a} \Ext_R^k(L,H_{\un+ \ur})$. Thus, $\bigoplus_{\un \in \bz^r}{\Ext_R^k(L,H_{\un})}$ is a $\bz^r$-graded module over the $\bn^r$-graded ring $\mathcal{R}({\bf I})$. Consequently,
    \begin{align*}
    \Ext_R^k \big(L, \mathcal{H} \big) = \bigoplus_{\un \in \bz^r}{\Ext_R^k(L, H_{\un})}.
    \end{align*}
    is a finitely generated $\bz^r$-graded $\mathcal{R}({\bf I})$-module. Similarly, $\bigoplus_{\un \in \bz^r} \Tor_k^R(L,H_{\un})$ is a finitely generated $\bz^r$-graded $\mathcal{R}({\bf I})$-module.
    
\end{para}

\begin{para}[{\bf $\bz \times \bz^r$-graded structure on Ext and Tor}]\label{para:multigrading}
    The multi-Rees ring $\mathcal{R}({\bf I}) = \bigoplus_{\un \in \bn^r} {\bf I}^{\un}$ can be considered as an $\bn \times\bn^r$-graded ring, where the $(l,\un)$ graded component of $\mathcal{R}({\bf I})$ is the $l$th graded component of ${\bf I}^{\un}$ for $(l,\un)\in \bn \times \bn^r$. Thus 
    \begin{align*}
    \mathcal{R}({\bf I}) 
    & = R[f_{i,j} : 1 \le i \le r, \; 1\le j \le g_i] \\
    & = R_{0}[x_1,\dots,x_d,f_{i,j} : 1 \le i \le r, \; 1\le j \le g_i],
    \end{align*}
    where $\deg(x_t)=(1, \underline{0})$ for $1 \le t \le d$, and $\deg(f_{i,j}) = (d_{i,j}, e_i)$ for $1 \le i \le r$, $1 \le j \le g_i$. Set ${\bf I}^{\un}M := 0$ for all $\un = (n_1, \dots, n_r) \in \bz^r$ with $n_i < 0$ for some $i$. Then the multi-Rees module $\mathcal{R}({\bf I},M)$ becomes a $\bz \times\bz^r$-graded $\mathcal{R}({\bf I})$-module, where $\mathcal{R}({\bf I},M)_{(l,\un)} = ({\bf I}^{\un}M)_l$ for every $(l,\un)\in \bz \times \bz^r$. Hence, setting 
    \begin{center}
    $\Ext_{R}^{k}\big(L, \mathcal{R}({\bf I},M)/ \mathcal{R}({\bf I},N)\big)_{(l,\un)} := \Ext_R^{k} \big(L, {\bf I}^{\un} M/ {\bf I}^{\un}N\big)_{l}$
    \end{center}
    for all $(l,\un) \in \bz \times \bz^r$, the module $\Ext_{R}^{k}\big(L, \mathcal{R} ({\bf I},M)/\mathcal{R}({\bf I},N)\big)$ becomes a $\bz \times \bz^r$-graded $\mathcal{R}({\bf I})$-module as well. Since we are only changing the grading from $\bz^r$ to $\bz\times \bz^r$, this module still remains finitely generated. Thus both 
    $$\Ext_{R}^{k}\big(L, \mathcal{R}({\bf I},M)/ \mathcal{R}({\bf I},N)\big) \text{ and } \Tor_{k}^{R}\big(L, \mathcal{R}({\bf I},M)/\mathcal{R}({\bf I},N)\big)$$ 
    are finitely generated $\bz \times \bz^r$-graded $\mathcal{R}({\bf I})$-modules.
\end{para}

In the following result, we show that for each fixed $k\ge 0$, the regularity of $\Ext_R^{k}(L, {{\bf I}^{\un}M}/{{\bf I}^{\un}N})$ is eventually supremum of finitely many  linear functions of $\un$.
    
\begin{proposition}\label{prop:v-Ext-lin}
With \textnormal{Setup~\ref{setup1}}, fix $k \in \bn$. Then, either $\Ext_R^{k} (L, {\bf I}^{\un}M/{\bf I}^{\un}N) = 0$ for all $\un \gg0$, or $\reg\big(\Ext_R^{k} (L, {\bf I}^{\un}M/{\bf I}^{\un}N) \big)$ is eventually supremum of finitely many linear functions of $\un$ with coefficient of $n_i$ in $\{ d_{i,1}, d_{i,2}, \dots,d_{i,g_{i}}\}$  for $1 \le i \le r$.
\end{proposition}
    
\begin{proof}
    In view of \ref{para:multigrading}, $\mscrm := \Ext_{R}^{k}\big(L, \mathcal{R}({\bf I},M)/\mathcal{R}({\bf I},N)\big)$ is a finitely generated $\bz \times \bz^r$-graded module over the $\bn \times \bn^r$-graded ring $S := \mathcal{R}({\bf I})$, where the multigrading of $S$ is shown in \ref{para:multigrading}. Note that $ \mscrm_{(\star,\un)} := \bigoplus_{l \in \bz}  \mscrm_{(l, \un)}$ is same as $\Ext_R^{k}(L, {\bf I}^{\un}M/{\bf I}^{\un}N)$ for each $\un \in \bz^r$. Hence, the proposition follows from Theorem~\ref{th:reggraded}.
\end{proof}

In a similar way, setting $\mscrm := \Tor_{k}^{R}\big(L, \mathcal{R}({\bf I},M)/\mathcal{R}({\bf I},N)\big)$, one obtains the counterpart of Proposition~\ref{prop:v-Ext-lin} for Tor modules.
    
\begin{proposition}\label{prop:v-Tor-lin}
     With \textnormal{Setup~\ref{setup1}}, fix $k \in \bn$. Then, either $\Tor_k^{R} (L, {\bf I}^{\un}M/{\bf I}^{\un}N) = 0$ for all $\un \gg 0$, or $\reg\big(\Tor_k^{R} (L, {\bf I}^{\un}M/{\bf I}^{\un}N) \big)$ is eventually supremum of finitely many linear functions of $\un$ with coefficient of $n_i$ in $\{ d_{i,1}, d_{i,2}, \dots,d_{i,g_{i}}\}$  for $1 \le i \le r$.
\end{proposition}

\section{Asymptotic regularity of $(U + {\bf I}^{\un}V)/{\bf I}^{\un}W$, and its consequences}\label{sec4}
In this section, we establish the main results. We first investigate the asymptotic behaviour of $\reg(M/{\bf I}^{\un}N)$, which is needed in the sequel. Throughout this section, we assume Setup~\ref{setup1}.

\begin{proposition}\label{prop:regquotient}
With \textnormal{Setup~\ref{setup1}}, let $M \neq 0$, and $\indeg(I_i) >0$ for some $i$. Then, exactly one of the following holds:
\begin{enumerate}[\rm(1)]
\item For all $\un \gg 0$, we have $M/{\bf I}^{\un}N = 0$.
\item For all $\un \gg 0$, we have ${\bf I}^{\un}N = 0$, and hence $\reg(M/{\bf I}^{\un}N) = \reg(M)$.
\item For all $\un \gg 0$, we have $M \supsetneqq {\bf I}^{\un} N \supsetneqq 0$. In this case,
$$ \reg(M/{\bf I}^{\un}N) = \reg({\bf I}^{\un} N) - 1,$$
which is eventually the supremum of finitely many linear functions of $\un$ with coefficient of $n_i$ in $\{ d_{i,1}, d_{i,2}, \dots,d_{i,g_{i}}\}$  for $1 \le i \le r$. 
\end{enumerate}
\end{proposition}

\begin{proof}
In view of \cite[Thm.~1.5.(ii)]{KS96}, $\Ass_R(M/{\bf I}^{\un}N)$ stabilizes for all $\un\gg 0$. It follows that either $M/{\bf I}^{\un}N = 0$ for all $\un \gg 0$ or $M/{\bf I}^{\un}N \neq 0$ for all $\un \gg 0$. Assume that $M/{\bf I}^{\un}N \neq 0$ for all $\un \gg 0$. Note that $\mathcal{R}({\bf I},N)$ is a finitely generated $\bn^r$-graded module over the Noetherian standard $\bn^r$-graded ring $\mathcal{R}({\bf I})$. Therefore by Lemma~\ref{lem:stab}, either ${\bf I}^{\un} N =0$ for all $\un \gg 0$, or $ {\bf I}^{\un} N \neq 0$ for all $\un \gg 0$. If ${\bf I}^{\un} N = 0$ for all  $\un \gg 0$, then
$$\reg(M/{\bf I}^{\un} N) = \reg(M) \; \text{ for all } \un \gg 0.$$ 
Thus, we may assume that ${\bf I}^{\un} N \neq 0$ for all $\un \gg 0$. Consider the short exact sequence
$$ 0 \to {\bf I}^{\un}N \to M \to M/{\bf I}^{\un}N \to 0. $$
By Lemma~\ref{lem:reglem}, we obtain
\begin{equation}\label{eq:regsesquotient}
\begin{aligned}
\reg(M/{\bf I}^{\un}N) &\le \max \{ \reg(M), \reg({\bf I}^{\un}N)-1 \} \quad \text{and} \\
\reg({\bf I}^{\un}N) &\le \max \{ \reg(M), \reg(M/{\bf I}^{\un}N)+1 \}.
\end{aligned}
\end{equation}
Now, in view of \ref{para:multigrading}, by taking $\mscrm = \mathcal{R}({\bf I}, N)$ and $S = \mathcal{R}({\bf I})$, from  Theorem~\ref{th:reggraded}, we obtain
\begin{equation} \label{eq:regineq1}
\reg({\bf I}^{\un}N) =  \sup \{ \mathscr{L}_{s}(\un) : 1 \le s \le m\} \; \mbox{ for all } \; \un \gg 0,
\end{equation}
where each $\mathscr{L}_s : \bn^r \to \bz$ is a linear function defined as $$\mathscr{L}_s(\underline{n}) = (w_{s,1}n_1 + w_{s,2}n_2 + \cdots + {w_{s,r}n_r}) + c_{s}$$ 
for some $w_{s,i}\in \{ d_{i,1},\dots,d_{i,g_{i}}\}$ and $c_{s} \in \bz$ for $1 \le s \le m$ and $1\le i\le r$. Since $\indeg(I_i)>0$ for some $i$, we have $d_{i,j}>0$ for all $1\le j\le g_i$ for this $i$, and hence $w_{s,i}>0$ for all $s$. Consequently, $\reg(M)<\reg({\bf I}^{\un}N)$ for all $\un\gg 0$. Therefore, both the inequalities in \eqref{eq:regsesquotient} yield that 
$$\reg(M/{\bf I}^{\un}N) = \reg({\bf I}^{\un}N) -1 \; \mbox{ for all }  \un \gg 0.$$
Consequently, the result follows from \eqref{eq:regineq1}.
\end{proof}


The following result strengthens \cite[Thm.~3.2]{TW05}.

\begin{proposition}\label{prop:regLM}
    With \textnormal{Setup~\ref{setup1}}, suppose $\indeg(I_i) >0$ for some $i$. Then, exactly one of the following holds:
    \begin{enumerate}[\rm(1)]
    \item For all $\un \gg 0$, we have ${\bf I}^{\un} M \subseteq L$, and hence $\reg(L + {\bf I}^{\un} M) = \reg(L)$.
    \item For all $\un \gg 0$, we have  ${\bf I}^{\un} M \nsubseteq L$. In this case, $ \reg(L + {\bf I}^{\un}M)$ is eventually the supremum of finitely many linear functions of $\un$ with the coefficient of $n_i$ in $\{ d_{i,1}, d_{i,2}, \dots,d_{i,g_{i}}\}$  for $1 \le i \le r$.
    \end{enumerate}
\end{proposition}

\begin{proof}
    By the multigraded version of Artin-Rees lemma (cf.~\cite[Thm.~17.1.6]{SH06}), there exists $\uc \in \bn^r$ such that
    $$(L+{\bf I}^{\un}M)/L \cong {\bf I}^{\un}M/(L \cap {\bf I}^{\un}M) \cong {\bf I}^{\un}M/ {\bf I}^{\un-\uc}(L \cap {\bf I}^{\uc}M) \; \text{ for all } \un \ge \uc.$$ 
    Note that
    $$ \bigoplus_{\un \ge \uc } \frac{{\bf I}^{\un}M}{{\bf I}^{\un- \uc}(L \cap {\bf I}^{\uc}M)} = \frac{\mathcal{R}({\bf I}, {\bf I}^{\uc}M)}{\mathcal{R}({\bf I},L \cap {\bf I}^{\uc}M)},$$
    and the module $\mathcal{R}({\bf I},{\bf I}^{\uc}M)/\mathcal{R}({\bf I}, L \cap {\bf I}^{\uc}M)$ is a finitely generated $\bz^r$-graded $\mathcal{R}({\bf I})$-module, where $\big(\mathcal{R}({\bf I}, {\bf I}^{\uc}M)/\mathcal{R}({\bf I}, L \cap {\bf I}^{\uc}M)\big)_{\un} = 0$ for all $\un = (n_1, \dots, n_r) \in \bz^r$ with $n_i < 0$ for some $i$. Therefore, by Lemma~\ref{lem:stab}, either $(L +{\bf I}^{\un}M)/L = 0$ for all $\un \gg 0$ or $(L +{\bf I}^{\un}M)/L \neq 0$ for all $\un \gg 0$. This implies that ${\bf I}^{\un} M \subseteq L$ for all $\un \gg 0$, or ${\bf I}^{\un} M \nsubseteq L$ for all $\un \gg 0$. If ${\bf I}^{\un}M \subseteq L $ for all $\un \gg 0$, then $\reg (L+ {\bf I}^{\un}M) = \reg (L) $ for all $\un \gg 0$. Assume that ${\bf I}^{\un} M \nsubseteq L$ for all $\un \ge \um$, where $\um \ge \uc$. Consider the short exact sequence
    \begin{equation*}
    0 \to L \to (L+{\bf I}^{\un}M) \to (L+{\bf I}^{\un}M)/L \to 0.
    \end{equation*} 
    Thus, by Lemma~\ref{lem:reglem}, we obtain
    \begin{equation}\label{eq:regsessum}
    \begin{aligned}
    & \reg(L+{\bf I}^{\un}M) \le \max \{ \reg(L), \reg((L+{\bf I}^{\un}M)/L) \}, \; \mbox{and} \\
    & \reg((L+{\bf I}^{\un}M)/L) \le \max \{ \reg(L)-1, \reg(L+{\bf I}^{\un}M) \}.
    \end{aligned}
    \end{equation}
    In view of \ref{para:multigrading},  $\mathcal{R}({\bf I},{\bf I}^{\uc}M)/\mathcal{R}({\bf I}, L \cap {\bf I}^{\uc}M)$ is a finitely generated $\bz \times \bz^r$-graded $\mathcal{R}({\bf I})$-module. Therefore, by Theorem~\ref{th:reggraded}, it follows that 
    \begin{equation}\label{eq:regeqLandM}
    \reg\big((L+{\bf I}^{\un}M)/L\big) = \reg\big({\bf I}^{\un}M/{\bf I}^{\un - \uc}(L \cap {\bf I}^{\uc} M)\big)
    \end{equation}
    is eventually the supremum of finitely many linear functions of $\un$ with the coefficient of $n_i$ in $\{ d_{i,1}, d_{i,2}, \dots,d_{i,g_{i}}\}$  for $1 \le i \le r$. Therefore, since $\indeg(I_i) > 0$ for at least one $i$, by a similar argument as before, $ \reg(L) < \reg((L+ {\bf I}^{\un}M)/L)$ for all $\un \gg 0$. 
    Hence, the inequalities in \eqref{eq:regsessum} yield that $ \reg(L+{\bf I}^{\un}M) = \reg((L+{\bf I}^{\un}M)/L)$ for all $\un\gg 0$. The assertion in (2) now follows from \eqref{eq:regeqLandM}.
\end{proof}
 
In the following result, we investigate $\reg\big( (U + {\bf I}^{\un} V)/ {\bf I}^{\un} W\big)$ for all $\un \gg 0$, under certain conditions.
    
\begin{theorem}\label{th:main1}
With \textnormal{Setup~\ref{setup1}}, let $\indeg(I_i) >0$ for some $i$. Let $U$, $V$ and $W$ be graded $R$-submodules of a finitely generated $\bz$-graded $R$-module $T$ such that $W \subseteq V$. Assume that one of the following conditions holds:
\begin{enumerate}[\rm (i)]
    \item $U$ is an $R_+$-torsion module, or
    \item $I^{\uc} V/ I^{\uc} W$ is an $R_+$-torsion module for some $ \uc \in \bn^r$.
\end{enumerate}
Then, either $(U + {\bf I}^{\un}V)/{\bf I}^{\un}W = 0 $ for all $\un \gg 0$, or $(U + {\bf I}^{\un}V)/{\bf I}^{\un}W \neq 0 $ for all $\un \gg 0$. In the second case, exactly one of the following holds:
\begin{enumerate}[\rm (1)]
    \item $\reg\big((U + {\bf I}^{\un}V)/{\bf I}^{\un}W\big) = \reg(U)$ for all $\un \gg 0$.
    \item $\reg\big((U + {\bf I}^{\un}V)/{\bf I}^{\un}W\big)$ is eventually the supremum of finitely many linear functions of $\un$ with the coefficient of $n_i$ in $\{d_{i,1}, d_{i,2}, \dots,d_{i,g_{i}}\}$ for $1 \le i \le r$. 
\end{enumerate}
\end{theorem}

\begin{proof}
For each $\un \ge \underline{0}$, we set
$$A_{\un} := \frac{U + {\bf I}^{\un}W}{{\bf I}^{\un}W}, \quad B_{\un} := \frac{{\bf I}^{\un}V}{{\bf I}^{\un}W}, \quad \text{and} \quad C_{\un} := \frac{U + {\bf I}^{\un}V}{{\bf I}^{\un}W}.$$
First, we record the following two auxiliary assertions that will be used in the proof. We prove these assertions after establishing the main conclusion of the theorem.
\begin{enumerate}[(I)]
\item \label{item:fact1}
Let $B_{\un} \neq 0$ for all $\un \gg 0$. Then $\reg( B_{\un} )$ is eventually the supremum of finitely many linear functions of $\un$ with the coefficient of $n_i$ in $\{d_{i,1}, d_{i,2}, \dots,d_{i,g_{i}}\}$  for $1 \le i \le r$.
\item \label{item:fact2}
Let $U \neq 0$, and $A_{\un} \neq 0$ for all $\un \gg 0$. Then, either $\reg(A_{\un}) = \reg(U)$ for all $\un \gg 0$, or $\reg(A_{\un})$ is eventually the supremum of finitely many linear functions of $\un$ with the coefficient of $n_i$ in $\{d_{i,1}, d_{i,2}, \dots,d_{i,g_{i}}\}$  for $1 \le i \le r$. 
\end{enumerate}

In view of the proof of \cite[Prop.~3.4]{KW04}, $\Ass_R(C_{\un})$ stabilizes for all $\un \gg 0$. Therefore, either $C_{\un} = 0 $ for all $\un \gg 0$, or $C_{\un} \neq 0 $ for all $\un \gg 0$. We may assume that $C_{\un} \neq 0 $ for all $\un \gg 0$. We need to show that the statement in (1) or the statement in (2) holds. Note that statements (1) and (2) cannot hold simultaneously, since $\indeg(I_i)>0$ for some $i$. If $U = 0$, then 
$ C_{\un} = B_{\un} \neq 0$ for all $\un \ge \underline{0}$.
In this case, by \eqref{item:fact1}, the statement in (2) holds. So, we may assume that $U \neq 0$. By Lemma~\ref{lem:stab}, $B_{\un} = 0$ for all $\un \gg 0$, or $B_{\un} \neq 0$ for all $\un \gg 0$. In the case when $B_{\un} = 0$ for all $\un \gg 0$, we have $ A_{\un} = C_{\un} \neq 0$ for all $\un \gg 0$.
Hence, by \eqref{item:fact2}, either (1) or (2) holds.

We now consider the remaining case, i.e., $U \neq 0$, and $B_{\un} \neq 0$ for all $\un \gg 0$. Note that $A_{\un} + B_{\un} = C_{\un}$. It produces a natural map $A_{\un} \oplus B_{\un} \to C_{\un}$ defined by $(a,b) \mapsto a+b$. This map induces the short exact sequence
\begin{equation}\label{eq:sesUVW}
    0 \to A_{\un} \cap B_{\un} \to A_{\un} \oplus B_{\un} \to C_{\un} \to 0.
\end{equation}
If $U$ is $R_+$-torsion, then for each $\un \ge \underline{0}$, the module 
$$A_{\un} = (U+{\bf I}^{\un}W)/{\bf I}^{\un}W \cong U/(U \cap {\bf I}^{\un}W)$$ 
is also $R_+$-torsion. On the other hand, if ${\bf I}^{\uc}V/{\bf I}^{\uc}W$ is $R_+$-torsion for some $\uc \in \bn^r$, then $B_{\un} = {\bf I}^{\un}V/{\bf I}^{\un}W$ is also $R_+$-torsion for all $\un \ge \uc$. Consequently, in either case, $A_{\un} \cap B_{\un}$ is an $R_+$-torsion submodule of $C_{\un}$ for all $\un \ge \uc$. Therefore, by Lemma~\ref{lem:reglem}.(4), we obtain
\begin{align}
    \reg(C_{\un}) &= \max\big\{ \reg(A_{\un} \cap B_{\un}), \;\reg\big(C_{\un}/(A_{\un} \cap B_{\un})\big) \big\} \nonumber \\
    &\ge \reg(A_{\un} \cap B_{\un}) \; \text{ for all } \un \ge \uc.\label{eq:reglbound}
\end{align} 
Moreover, applying Lemma~\ref{lem:reglem}.(4) to the short exact sequence \eqref{eq:sesUVW}, one has
\begin{equation}\label{reg-A-B-C}
    \reg\left(A_{\un} \oplus B_{\un}\right) = \max \left\{\reg(A_{\un}\cap B_{\un}) , \reg(C_{\un}) \right\} \; \text{ for all } \un \ge \uc.
\end{equation}
Combining \eqref{eq:reglbound} and \eqref{reg-A-B-C}, it follows that
\begin{equation}\label{eq:regeqCn}
    \reg(C_{\un}) = \reg\left(A_{\un} \oplus B_{\un}\right) = \max \left\{ \reg(A_{\un}), \reg (B_{\un}) \right\} \; \text{ for all } \un \ge \uc.
\end{equation}
Note that either $A_{\un}=0$ for all $\un\gg 0$, or $A_{\un}\neq 0$ for all $\un\gg 0$. In the first case, $\reg(A_{\un}) = -\infty$ for all $\un\gg 0$, which implies, by \eqref{eq:regeqCn}, that $\reg(C_{\un}) = \reg (B_{\un})$ for all $\un\gg 0$. By \eqref{item:fact1}, $\reg(B_{\un})$ is eventually the supremum of finitely many linear functions of $\un$ with the coefficient of $n_i$ in $\{d_{i,1}, d_{i,2}, \dots,d_{i,g_{i}}\}$ for $1 \le i \le r$. In the other case, i.e., when $A_{\un}\neq 0$ for all $\un\gg 0$, by \eqref{item:fact2}, either $ \reg(A_{\un}) = \reg(U)$ for all $\un \gg 0$, or $\reg(A_{\un})$ is eventually the supremum of finitely many linear functions of $\un$ with the coefficient of $n_i$ in $\{d_{i,1}, d_{i,2}, \dots,d_{i,g_{i}}\}$. Therefore, in either case, by \eqref{eq:regeqCn}, the statement in (2) holds. Thus, it remains to establish the auxiliary assertions \eqref{item:fact1} and \eqref{item:fact2}.

\textit{Proof of }\eqref{item:fact1}: In view of \ref{para:multigrading}, the module $\mscrm := \bigoplus_{\un \in  \bn^r}{\bf I}^{\un}V/ {\bf I}^{\un}W$ is a finitely generated $\bz \times \bz^r$-graded module over the $\bn \times \bn^r$-graded Rees ring $S := \mathcal{R}({\bf I})$. Hence, the desired conclusion follows from Theorem~\ref{th:reggraded}.

\textit{Proof of }\eqref{item:fact2}: By the multigraded version of Artin-Rees lemma, there exists $\um \in \bn^r$ such that
$$ A_{\un} = \frac{U + {\bf I}^{\un}W}{{\bf I}^{\un}W} \cong \frac{U}{U \cap {\bf I}^{\un}W} \cong \frac{U}{{\bf I}^{\un - \um}(U \cap {\bf I}^{\um}W)} \; \text{ for all } \un \ge \um.$$
Therefore, the conclusion of \eqref{item:fact2} follows immediately from Proposition~\ref{prop:regquotient}.
\end{proof}

\begin{remark}\label{rmk-quot-main-res}
In Theorem~\ref{th:main1}, we observe the following:
\begin{enumerate}[\rm (1)]
    \item 
    The conclusion of Theorem~\ref{th:main1} holds whenever the module $U$ has finite length.
    \item \label{rmk-2}
    Set $C_{\un} := (U + {\bf I}^{\un}V)/{\bf I}^{\un}W$ for $\un\ge \underline{0}$. Then either $C_{\un} = 0 $ for all $\un \gg 0$, or $C_{\un} \neq 0 $ for all $\un \gg 0$. In the second case, if the ideal $I_i$ is generated in a single degree, say $d_i$, for all $1 \le i \le r$, then either $\reg(C_{\un}) = \reg(U)$ for all $\un \gg 0$, or $\reg(C_{\un})$ is a linear function of $\un$, where the coefficient of $n_i$ is $d_i$.
    \item 
    In the case of a single ideal $I$, the conclusion as in Remark~\ref{rmk-quot-main-res}.\eqref{rmk-2} holds even when the ideal $I$ is not generated in a single degree.
\end{enumerate}
    
\end{remark}

Now we are in a position to achieve the main results.

\begin{theorem}\label{th:mainext}
With \textnormal{Setup~\ref{setup1}}, let $\indeg(I_i) >0$ for some $i$. Fix $k \ge 0$. Set $G := \Ext_R^k(L,M)$ and $G_{\un} := \Ext_R^k(L, M/{\bf I}^{\un} N)$ for all $\un \ge \underline{0}$. Assume that at least one of the following two conditions holds.
\begin{enumerate}[\rm (i)]
    \item $G$ is $R_+$-torsion.
    \item $G_{\un}$ is $R_+$-torsion for all $\un \gg 0$.
\end{enumerate}

Then, either $G_{\un} = 0$ for all $\un \gg 0$, or $G_{\un} \neq 0$ for all $\un \gg 0$. In the second case, exactly one of the following holds:
\begin{enumerate}[\rm (1)]
    \item $\reg(G_{\un}) = \reg(G)$ for all $\un \gg 0$.
    \item $\reg(G_{\un})$ is eventually the supremum of finitely many linear functions of $\un$ with coefficient of $n_i$ in $\{d_{i,1}, d_{i,2}, \dots,d_{i,g_{i}}\}$  for $1 \le i \le r$. 
\end{enumerate}
\end{theorem}

\begin{proof}
In view of Lemma~\ref{lem:ext:interpret}, there exists $\uc \in \bn^r$ such that
$$ \Ext_{R}^{k}(L, M/ {\bf I}^{\un}N) \cong (U + {\bf I}^{\un - \uc}V)/ {\bf I}^{\un - \uc}W \; \text{ for all } \un \ge \uc, $$
where $U$, $V$ and $W$ are graded submodules of a finitely generated graded $R$-module $T$ satisfying $W \subseteq V$ and $U\cong \Ext_R^k(L,M)$. Note that $U \cong G$. Moreover, if $G_{\un}$ is $R_+$-torsion for all $\un\gg0$, then $(U+{\bf I}^{\un}V)/{\bf I}^{\un}W$ is $R_+$-torsion for all $\un\gg0$. Consequently, ${\bf I}^{\un}V/{\bf I}^{\un}W$ is also $R_+$-torsion for all $\un\gg0$. Therefore, the desired result follows immediately from Theorem~\ref{th:main1}.
\end{proof}

Similar to the proof of Theorem~\ref{th:mainext}, using Lemma~\ref{lem:tor:interpret} and Theorem~\ref{th:main1}, we obtain the following.

\begin{theorem}\label{th:maintor}
With \textnormal{Setup~\ref{setup1}}, let $\indeg(I_i) >0$ for some $i$. Fix $k \ge 0$. Set $G := \Tor_k^R(L,M)$ and $G_{\un} := \Tor_k^R(L, M/{\bf I}^{\un} N)$ for all $\un \ge \underline{0}$. Suppose any one of the following two conditions holds.
\begin{enumerate}[\rm (i)]
    \item $G$ is $R_+$-torsion.
    \item $G_{\un}$ is $R_+$-torsion for all $\un \gg 0$.
\end{enumerate}
    
Then, either $G_{\un} = 0$ for all $\un \gg 0$, or $G_{\un} \neq 0$ for all $\un \gg 0$. In the second case, exactly one of the following holds:
\begin{enumerate}[\rm (1)]
    \item $\reg(G_{\un}) = \reg(G)$ for all $\un \gg 0$.
    \item $\reg(G_{\un})$ is eventually the supremum of finitely many linear functions of $\un$ with coefficient of $n_i$ in $\{d_{i,1}, d_{i,2}, \dots,d_{i,g_{i}}\}$  for $1 \le i \le r$. 
\end{enumerate}
\end{theorem}

\begin{remark}\label{rmk-v-num}
\begin{enumerate}[(1)]
    \item In Theorem~\ref{th:mainext}, suppose that $G_{\un}\neq 0$ is $R_+$-torsion for all $\un \gg 0$. Then $\depth(R_+, G_{\un}) = 0$ for all $\un \gg 0$. Hence, by \cite[Prop.~2.2]{FG25a}, we have
    $$v(G_{\un})\le\reg(G_{\un}) \; \mbox{ for all } \un \gg 0,$$
    where $v(G_{\un})$ denotes the Vasconcelos invariant (or $v$-number) of $G_{\un}$.
    \item 
    A similar remark as in \ref{rmk-v-num}.(1) can be made for Theorem~\ref{th:maintor} as well.
\end{enumerate}
\end{remark}

\section{Examples}\label{sec:examples}

In this section, we provide some examples to complement the main results. In the first example, for each fixed $k \ge 1$, we have $\Ext^{2k-1}_R(L,M)=0$ and $\Tor_{2k}^R(L,M)=0$, while the functions $$\reg(\Ext_R^{2k-1}(L,M/I^nN)) \quad \text{and} \quad \reg(\Tor_{2k}^R(L,M/I^nN))$$ 
are eventually linear in $n$. Note that $\Ext^{2k}_R(L,M)\neq 0$ for all $k \ge 0$. However, the function $\reg(\Ext_R^{2k}(L,M/I^nN))$ can be asymptotically either constant or linear with a non-zero leading coefficient: it is eventually constant, with value $\reg(\Ext_R^{2k}(L,M))$, for $k\ge 1$, whereas it is eventually linear with a non-zero leading coefficient for $k=0$.

\begin{example}\label{example2}
    Let $R = K[X,Y]/\langle XY\rangle$, where $K$ is a field, and $\deg(X) = \deg(Y) = 1$. Suppose $x$ and $y$ are the residue classes of $X$ and $Y$ in $R$, respectively. Set $L := R/\langle x\rangle$, $M = N = I := \langle y \rangle$. Then, $\Ext_R^k(L,M/I^nN)$ and $\Tor_k^R(L,M/I^nN)$ are $R_+$-torsion for all $k \ge 0$, and $n \ge 1$. Moreover, the following hold.
    \begin{enumerate}[(1)]
    \item When $k \ge 1$ is odd, $\Ext_R^{k}( L, M) =0$. In other cases, we have
    $$\reg(\Ext_R^{k}( L, M)) = 
    \left\{ \begin{array}{ll}
    0 &\text{ if } k =0, \\
    1-k &\text{ if } k \ge 2 \text{ is even.}
    \end{array} \right.$$    
    \item 
    When $k \ge 2$ is even, $\Tor_k^R(L,M) = 0$. In other cases, we have
    $$\reg(\Tor_k^{R}( L, M)) = 
    \left\{ \begin{array}{ll}
    0   &\text{ if } k =0, \\
    1+k &\text{ if } k \ge 1 \text{ is odd.}
    \end{array}\right.$$    
    \item 
    For all $n \ge 1$, $\reg\big(\Ext_R^{k}( L, M/I^nN)\big) = 
    \left \{  
    \begin{array}{ll} 
     n-k& \text{if $k=0$, or if $k\ge 1$ is odd}  \\ 
     1-k& \text{if $k\ge 2$ is even}.
    \end{array}
    \right.$
    \item 
    For all $k\ge 0$ and $n \ge 1$,
    $ 
    \reg\big(\Tor_k^{R}( L, M/I^nN)\big) = 
    \left \{  
    \begin{array}{ll} 
    n+k& \text{if $k$ is even}\\     
    1+k& \text{if $k$ is odd}. 
    \end{array}
    \right.$
    \end{enumerate}
\end{example}

\begin{proof}
Note that 
    \begin{equation}\label{eq:reso} 
    \mathbb{F}_{L}: \quad \cdots \xlra{y} R(-3) \xlra{x} R(-2) \xlra{y} R(-1) \xlra{x} R \lra 0
    \end{equation}
    is the graded minimal free resolution of $L$ over $R$. Since 
    \begin{equation}\label{eq}
    \Hom_{R}(R(-k),W) \cong W(+k) \quad \mbox{and} \quad  R(-k) \otimes_{R}W \cong W(-k)
    \end{equation} 
    for any graded $R$-module $W$ and for all $k \in \bz$, it follows that 
    \begin{align} \Hom_{R}(\mathbb{F}_{L},W): \quad 0\lra W \xlra{x} W(+1) \xlra{y} W(+2) \xlra{x} W(+3) \xlra{y} \cdots \label{Hom-L-reso}\\
     \mbox{and~} \mathbb{F}_{L} \otimes_{R} W: \quad \cdots \xlra{y} W(-3) \xlra{x} W(-2) \xlra{y} W(-1) \xlra{x} W \lra 0.\label{Tensor-L-reso}
     \end{align}
    Putting $W = M$ in \eqref{Hom-L-reso} and \eqref{Tensor-L-reso}, and then computing (co)homologies, we obtain
    \begin{align*}
    \Ext_R^k(L,M) & =
    \left \{
    \begin{array}{ll}
    M & \mbox{if~} k=0,  \\
    0 & \mbox{if $k\ge 1$ is odd},\\
    \big(\langle y \rangle/\langle y^2 \rangle\big)(+k) & \mbox{if $k\ge 2$ is even}
    \end{array}
    \right. \\ 
    \mbox{and} \quad 
    \Tor_k^R(L,M) & =
    \left \{
    \begin{array}{ll}
    M & \mbox{if~} k=0,  \\
    \big(\langle y\rangle/\langle y^2 \rangle\big)(-k) & \mbox{if $k\ge 1$ is odd},\\
    0 & \mbox{if $k\ge 2$ is even}.
    \end{array}
    \right.    
    \end{align*}
    Set $S:=K[X,Y]$. Then $M = \langle y \rangle \cong R/\langle x \rangle \cong S/\langle X \rangle$. It follows that $\reg_R(M) = \reg_S(M) = 0$. Note that $M/I^nN \cong \langle y\rangle/\langle y^{n+1} \rangle$ for all $n \ge 1$. Therefore, by putting $W = \langle y \rangle/\langle y^{n+1} \rangle$ in \eqref{Hom-L-reso} and \eqref{Tensor-L-reso}, and then computing (co)homologies, for all $n \ge 1$, one has that
    \begin{align*}
    \Ext_R^k(L,M/I^nN) & =
    \left \{
    \begin{array}{ll}
    \langle y\rangle/\langle y^{n+1} \rangle & \mbox{if~} k=0,  \\
    \big(\langle y^n\rangle/\langle y^{n+1} \rangle\big)(+k) & \mbox{if $k\ge 1$ is odd},\\
    (\big\langle y\rangle/\langle y^{2} \rangle\big)(+k) & \mbox{if $k\ge 2$ is even}
    \end{array}
    \right. \\ 
    \mbox{and} \quad 
     \Tor_k^R(L, M/I^nN) & =
    \left \{
    \begin{array}{ll}
    \langle y\rangle/\langle y^{n+1} \rangle & \mbox{if~} k=0,  \\
    \big(\langle y\rangle/\langle y^{2} \rangle\big)(-k) & \mbox{if $k\ge 1$ is odd},\\
    \big(\langle y^n\rangle/\langle y^{n+1} \rangle\big)(-k) & \mbox{if $k\ge 2$ is even}.
    \end{array}
    \right.    
    \end{align*}
Clearly, these modules are $R_+$-torsion. Using the decompositions
\begin{align*}
    \langle y\rangle/\langle y^{n+1} \rangle \cong Ky \oplus \cdots \oplus Ky^{n} \quad \text{and} \quad \langle y^n \rangle / \langle y^{n+1} \rangle \cong Ky^n
\end{align*}
one computes $\reg(\langle y \rangle /\langle y^{n+1} \rangle) = n$ and $\reg(\langle y^n \rangle / \langle y^{n+1} \rangle) = n$ for all $n\ge 1$. Since $\reg(W(+k)) = \reg(W) - k$, all the conclusions hold.
\end{proof}

The following example shows that $\reg\big(\Ext^0_R(L, M/{\bf I}^{\un}N)\big)$ is eventually the supremum of finitely many linear functions $\un$, which need not be linear in $\un$.

\begin{example}\label{exam:non-linear}
    Let $R$ and $L$ be as in Example~\ref{example2}. Set $M=N:=R$, $I:= \langle x^2,y \rangle$ and $J := \langle x,y^2\rangle$. Then, $\Ext_R^k(L,M/I^mJ^nN)$ and $\Tor_k^R(L,M/I^mJ^nN)$ are $R_+$-torsion for all $k \ge 0$, and $m,n \ge 1$. Moreover, we have the following:
    \begin{enumerate}[(1)]
    \item 
    For all $m,n \ge 1$,    
    \newline
    $\reg\left(\Ext_R^{k}\left( L, \frac{M}{I^m J^nN}\right)\right) = 
    \left \{  
    \begin{array}{lll} 
     \max\{2m+n-1, m + 2n-1\} & \text{if } k =0, \\
     m +2n -1 -k& \text{if $k\ge 1$ is odd},  \\ 
     2m + n - 1 -k& \text{if $k\ge 2$ is even}.
    \end{array}
    \right.$
    \item 
    For all $k\ge 0$ and $m,n \ge 1$,
    \newline
    $ 
    \reg\left(\Tor_k^{R} \left( L, \frac{M}{I^m J^n N}\right)\right) = 
    \left \{  
    \begin{array}{ll} 
    m + 2n +k -1 & \text{if $k$ is even},\\     
    2m + n + k-1 & \text{if $k$ is odd}. 
    \end{array}
    \right.$
    \end{enumerate}
\end{example}

\begin{proof}
    Let $m,n \ge 1$. Then
    $I^m = \langle x^{2m},y^m \rangle$ and $J^n = \langle x^n,y^{2n} \rangle $. So
    $$M/I^m J^nN = R/I^m J^n = R/\langle x^{2m+n},y^{m +2n} \rangle.$$
    Putting $W = R/I^m J^n$ in \eqref{Hom-L-reso} and \eqref{Tensor-L-reso}, and then computing (co)homologies,
    \begin{align*}
    \Ext_R^k \left(L,\frac{M}{I^m J^nN}\right) = 
    \left \{
    \begin{array}{ll}
    \langle x^{2m+n-1},y\rangle/\langle x^{2m+n},y^{m +2n}\rangle & \mbox{if~} k=0  \\
    \big(\langle x,y^{m+2n-1}\rangle /\langle x,y^{m +2n}\rangle \big)(+k) & \mbox{if~} k\ge 1 \mbox{~is odd} \\
    \big( \langle x^{2m+n-1},y\rangle/\langle x^{2m+n},y\rangle \big)(+k) & \mbox{if~} k\ge 2 \mbox{~is even}
    \end{array}
    \right. \\
    \mbox{and }\;
    \Tor_k^R \left(L, \frac{M}{I^m J^nN} \right) =
    \left \{
    \begin{array}{ll}
    R/\langle x,y^{m +2n}\rangle & \mbox{if~} k=0  \\
    \big( \langle x^{2m+n-1},y\rangle / \langle x^{2m+n},y\rangle \big)(-k) & \mbox{if~} k\ge 1 \mbox{~is odd} \\
    \big(\langle x,y^{m + 2n -1}\rangle/ \langle x,y^{m +2n}\rangle\big)(-k) & \mbox{if~} k\ge 2 \mbox{~is even}
    \end{array}
    \right.
    \end{align*}
    for all $m,n \ge 1$. By similar computations as in Example~\ref{example2}, one obtains (1) and (2).
\end{proof}

We end this article by posing the following natural questions.

\begin{question}
Does the conclusion of Theorem~\ref{th:main1} still hold if we remove the conditions that either $U$ is $R_+$-torsion or ${\bf I}^{\uc}V/{\bf I}^{\uc}W$ is $R_+$-torsion for some $ \uc \in \bn^r$?
\end{question}

\begin{question}
    Does the conclusion of Theorem~\ref{th:main} still hold if we remove the conditions either $G$ is $R_+$-torsion or $G_{\un}$ is $R_+$-torsion for all $ \un \gg 0$?
\end{question}

\bibliographystyle{plain}
\bibliography{ref}

\end{document}